\documentclass[11pt]{amsart}
\usepackage{graphicx}
\usepackage{amssymb}
\newcommand{\h}{\mathcal{H}}
\newcommand{\NN}{\mathbb N}

\newcommand{\RR}{\mathbb R}
\newcommand{\CC}{\mathbb C}
 \newtheorem{thm}{Theorem}[section]
 
 \newtheorem{lem}[thm]{Lemma}
 \newtheorem{prop}[thm]{Proposition}
 \theoremstyle{definition}
 \newtheorem{defn}[thm]{Definition}
 \theoremstyle{remark}
 
 \newtheorem*{ex}{Example}
 \numberwithin{equation}{section}

\begin{document}

%
%
%
%
%
%
%
%
%
\title[Construction of Continuous Frames in Hilbert spaces]
 {Construction of Continuous Frames in Hilbert spaces}
\author[A. Rahimi, B. Daraby and Z. Darvishi]{A. Rahimi, B. Daraby and Z. Darvishi }
\address{%
Department of Mathematics\\
University of Maragheh\\
P.O. Box 55181-83111\\
Maragheh\\
Iran} \email{rahimi@maragheh.ac.ir, bdaraby@maragheh.ac.ir,darvishi\_z@ymail.com}

\subjclass[2010]{Primary 42C40; 4210 Secondary 41A58}

\keywords{Frame, Bessel sequence, Continuous frame,  Measure space, Wavelet frame,
Short-time Fourier transform, Gabor frame.}


\begin{abstract}
Extending the concept of frame  to continuous frame, in this
manuscript we will show that under certain conditions on the measure
of $\Omega$ and the dimension of $\h$ we can construct continuous
frames. Also, some examples are given.
\end{abstract}

\maketitle
\section{Introduction}

A discrete frame is a countable family of elements in a separable
Hilbert space which allows stable but not necessarily unique
decomposition of arbitrary elements into expansion of the frame
elements. The concept of generalization of frames was proposed by G.
Kaiser \cite{Gk} and independently by Ali, Antoine and Gazeau
\cite{Ali2} to a family indexed by some locally compact space
endowed with a Radon measure. These frames are known as continuous
frames. Gabardo and Han in \cite{Gb} called these frames
\textit{Frames associated with measurable spaces}, Askari-Hemmat,
Dehghan and Radjabalipour in \cite{Ra} called these frames
\textit{generalized frames}  and in mathematical physics are
referred to \textit{Coherent states}\cite{Ali2}.

 For more studies on continuous frames and its applications, the interested reader can refer to
\cite{Ali1, Ali2, Ra, chui, For, Gb}. In this paper, we focus on
positive measures and separable complex Hilbert spaces.

Wavelet and Gabor frames are used very often in signal processing
algorithms. Both systems are derived from a continuous transform,
which can be seen as a continuous frame \cite{Ali1, fuhr,Gro}.

\section{Continuous frames}
Thought this paper, $\h$ is a separable Hilbert space and $(\Omega
,\mu)$ is a measure space with positive measure $\mu$.
\begin{defn}
Let $\h $ be a complex Hilbert space and $(\Omega ,\mu)$ be a
measure space with positive measure $\mu.$ The mapping
$F:\Omega\to\h$ is called a continuous
frame with respect to $(\Omega ,\mu)$, if\\
\begin{enumerate}
\item  $F$ is weakly-measurable, i.e., for all $f\in \h$, the function
$\omega\to\langle f,F({\omega})\rangle$
is a measurable function on $\Omega$; \\
\item  there exist constants $A, B> 0$ such that
\begin{equation}\label{deframe}
A\|f\|^{2}\leq \int_{\Omega}|\langle
f,F({\omega})\rangle|^{2}\,d\mu(\omega) \leq B\|f\|^{2}, \quad (
f\in \h).
\end{equation}
\end{enumerate}
 The constants $A $ and $B$ are called continuous
frame bounds. $F$ is called a tight continuous frame if $A=B$ and
Parseval if $A=B=1$. The mapping $F$ is called  \emph{Bessel} if the
second inequality in (\ref{deframe}) holds. In this case, $B$ is
called the \emph{Bessel constant.}
\end{defn}

If $\mu$ is  counting measure and $\Omega=\NN$ then $F$ is a
discrete frame. In this sense continuous frames are the more general
setting.

The first inequality in (\ref{deframe}), shows that $F$ is complete,
i.e.,
$$\overline{\textrm{span}}\{F({\omega})\}_{\omega\in\Omega}=\h.$$

Like orthonormal bases we have the following proposition.
\begin{prop}
Let $F:\Omega\rightarrow\h$ be a continuous Bessel function,
$\Lambda\subseteq\Omega$ and $f\in \h$. Then $E_\Lambda:=\{\omega:
\omega\in\Lambda , \langle f, F(\omega)\rangle\neq 0\}$ is
$\sigma$-finite.
\end{prop}
\begin{proof} For $n\in \NN$ and $f\in\h$, let
$$K_n=\{\omega: \omega\in \Lambda, |\langle f,
F(\omega)\rangle|\geq\frac{1}{n}\},$$ then $E_\Lambda=\bigcup_n
K_n$. So
\begin{eqnarray*}
\frac{1}{n^2}\mu(K_n)&\leq& \int_{K_n}|\langle
f,F({\omega})\rangle|^{2}\,d\mu(\omega)\\
&\leq&\int_{E_\Lambda}|\langle
f,F({\omega})\rangle|^{2}\,d\mu(\omega)\\
&=&\int_{\Omega}|\langle f,F({\omega})\rangle|^{2}\,d\mu(\omega)\\
&\leq& B\|f\|^{2}.
\end{eqnarray*}

Therefore $\mu(K_n)<\infty$. Hence $E_\Lambda$ is $\sigma$-finite.
\end{proof}
\par

Let $F$ be a continuous frame with respect to $(\Omega ,\mu)$, then
the mapping
$$\Psi : \h\times\h \to \CC$$ defined by
$$\Psi(f,g)= \int_{\Omega}\langle f,F({\omega})\rangle\langle
F({\omega}),g\rangle \,d\mu(\omega), \quad ( f,g\in\h )$$ is well
defined, sesquilinear and bounded. By Cauchy-Schwarz's inequality, we
get
\begin{eqnarray*}|\Psi(f,g)|&\leq&
\int_{\Omega}|\langle f,F(\omega)\rangle\langle
F(\omega),g\rangle|\, d\mu(\omega)\\
&\leq& \left(\int_{\Omega}|\langle f,F(\omega)\rangle|^{2}
\,d\mu(\omega)\right)^{\frac{1}{2}}\left(\int_{\Omega}|\langle
F(\omega), g \rangle|^{2}\, d\mu(\omega)\right)^{\frac{1}{2}}\\
&\leq& B\| f\|\| g\|.
\end{eqnarray*}

Hence $\|\Psi\|\leq B.$ It follows that there exists a unique bounded operator ( Riesz
Representation Theorem) $S_{F}: \h\to\h$ such that
$$\Psi(f,g)=\langle S_{F}f, g\rangle,\quad ( f,g\in\h )$$ and moreover
$\|\Psi\|=\|S_F\|.$\\
Since $\langle S_{F}f,f\rangle=\int_{\Omega}|\langle
f,F(\omega)\rangle |^{2}\,d\mu(\omega),$ then $S_{F}$ is positive
and $AI\leq S_{F}\leq BI$. Hence $S_{F}$ is invertible, positive and
$\frac{1}{B} I\leq S^{-1}_{F}\leq \frac{1}{A} I$. We call $S_{F}$
the continuous frame operator of $F$ and we use the notation
$$S_{F}f=\int_{\Omega}\langle f,F(\omega)\rangle F(\omega)
\,d\mu(\omega)\quad (f\in\h),$$ which is valid in the weak sense.
Thus, every $f\in\h$ has the representations
$$f=S_{F}^{-1}S_{F}f=\int_{\Omega}\langle f, F(\omega)\rangle
S_{F}^{-1}F(\omega)\,d\mu(\omega)$$$$f=S_{F}S_{F}^{-1}f=\int_{\Omega}\langle
f, S_{F}^{-1}F(\omega)\rangle F(\omega)\,d\mu(\omega).$$
\begin{thm}\label{TF}\cite{RaNaDe}
Let $(\Omega, \mu)$ be a measure space and let $F$ be a Bessel
mapping from $\Omega$ to $\h.$ Then the operator
$T_{F}:L^{2}(\Omega, \mu)\to\h$ weakly defined  by $$\langle
T_{F}\varphi, h\rangle=\int_{\Omega}\varphi(\omega)\langle
F(\omega),h \rangle \,d\mu(\omega),\quad ( h\in\h )$$ is well
defined, linear, bounded and its adjoint is given by $$ T_{F}^{*}:
\h\to L^{2}(\Omega, \mu),\quad (T_{F}^{*}h)(\omega)=\langle h,
F(\omega)\rangle,\quad ( \omega\in\Omega ).$$ The operator $T_{F}$
is called the \textit{pre-frame operator or synthesis operator} and
$T_{F}^{*}$ is called the \textit{analysis operator} of $F$.
\end{thm}
The converse of Theorem \ref{TF} holds when $\mu$ is a
$\sigma$-finite measure \cite{RaNaDe}.
\begin{prop}\label{COTF}
Let $ (\Omega, \mu)$ be a measure space, where $\mu$ is a
$\sigma$-finite measure and $ F:\Omega\to\h$ be a measurable
function. If the mapping $ T_{F}: L^{2}(\Omega, \mu)\mapsto \h$
defined by
$$\langle T_{F}\varphi, h\rangle=\int_{\Omega}\varphi(\omega)\langle
F(\omega), h\rangle \,d\mu(\omega), \quad( \varphi\in L^{2}(\Omega,
\mu), h\in\h )$$ is a bounded operator then $ F$ is Bessel.
\end{prop}

The next theorem gives an equivalent characterization of continuous
frame \cite{RaNaDe}.
\begin{thm}\label{Tonto}
Let $(\Omega, \mu)$ be a measure space where $\mu $ is a
$\sigma$-finite measure. The mapping $F:\Omega\to\h$ is a continuous
frame with respect to $(\Omega, \mu)$ for $\h$ if and only if the
operator $T_{F}$ as defined in Theorem \ref{TF} is a bounded and
onto operator.
\end{thm}

Such as in the discrete case we have the next lemma.
\begin{lem}\cite{RaNaDe}
Let $F:\Omega\to\h$ be a Bessel function with respect to
$(\Omega,\mu)$. By the above notations $S_{F}=T_{F}T_{F}^{*}.$
\end{lem}
The following proposition is a criterion for a continuous frame for
a closed subspace of $\h$ being a continuous frame. For
discrete case see [\ref{C4}, Lemma 5.2.1]
\begin{prop}\cite{RaNaDe}
Suppose that $F$ is a continuous frame with respect to
$(\Omega,\mu)$ for a closed subspace $K$ of $\h$, where $\mu$ is a
$\sigma$-finite measure. Then $F$ is a continuous frame for $\h $ if
and only if $T^{*}$ is injective.

\end{prop}
It is well known that discrete Bessel sequences in a Hilbert space
are norm bounded above: if
$$ \sum_n|\langle f,f_n\rangle|^2\leq B\parallel f\parallel^2$$  for all $ f\in\h$, then $ \parallel f_n\parallel\leq\sqrt{B}$
for all $n$. For continuous Bessel mappings, the following example
shows that, it is possible to make a continuous Bessel mapping which
is unbounded.

\begin{ex}\label{PP}
Take an (essentially) unbounded (Lebesgue) measurable function $a :
\RR\rightarrow \CC$ such such that $a\in  L^2(\RR) \setminus
L^\infty(\RR)$. It is easy to see that such functions indeed exist;
consider for example the function
$$b(x):=\frac{1}{\sqrt{|x|}}, 0<|x|<1, \quad b(x)=\frac{1}{|x|^2}, |x|\geq 1\quad and \quad b(x)=0, x=0.$$ This function is clearly in $L^1(\RR) \setminus L^\infty(\RR)$
 and furthermore, $b(x)\geq 0$ for all $x\in \RR$. Now take $a(x)=\sqrt{b(x)}$. Choose a fixed vector $h\in\h, h \neq 0$. Then, the mapping
 $$F:\RR\rightarrow\h, \omega\mapsto F(\omega)=a(\omega)h$$     is weakly (Lebesgue) measurable and a continuous Bessel mapping, since
$$\int_{\Omega}|\langle
f,F({\omega})\rangle|^{2}\,d\mu(\omega) \leq \|h\|^2
\|a\|^2\|f\|^{2}$$ for all $f\in\h$, but $\|F(\omega)\|$ is
unbounded, since $a$ is unbounded.

\end{ex}
The following example shows that: even continuous frames need not
necessarily norm bounded.

\begin{ex}
Let $F:\RR\rightarrow \h$ be a norm-unbounded continuous Bessel
mapping with Bessel constant $B_1$ ( like example \ref{PP}) and
$G:\RR\rightarrow\h$ be a norm-bounded continuous frame with bounds
$A_2\leq B_2$ and also assume that $B_1<A_2$. Then
$F-G:\RR\rightarrow \h$ is a norm-unbounded frame. It is clear that
for any $f\in\h$
\begin{eqnarray*}
\int_{\Omega}|\langle f,F(\omega)-G(\omega)\rangle
|^2d\mu(\omega)&\leq&\int_{\Omega}|\langle f,F(\omega)\rangle
|^2d\mu(\omega)+\int_{\Omega}|\langle f,G(\omega)\rangle
|^2d\mu(\omega)\\
&\leq& (B_1+B_2)\|f\|^2.
\end{eqnarray*} So $F-G$ is a continuous Bessel mapping with bound $
B_1+B_2$. For the lower bound, observe that
\begin{eqnarray*}
\left(\int_{\Omega}|\langle f,F(\omega)-G(\omega)\rangle|^{2}
\,d\mu(\omega)\right)^{\frac{1}{2}}&\geq&\left(\int_{\Omega}|\langle
f,G(\omega)\rangle|^{2}
\,d\mu(\omega)\right)^{\frac{1}{2}}-\left(\int_{\Omega}|\langle
f,F(\omega)\rangle|^{2}
\,d\mu(\omega)\right)^{\frac{1}{2}}\\
&\geq&(\sqrt{A_2}-\sqrt{B_1})\|f\|
\end{eqnarray*}
and the lower bound is established. The mapping $F-G$ is not norm
bounded, since $$\|F(\omega)-G(\omega)\|\geq \|F(\omega)\|-M$$ where
$\|G(\omega)\|\leq M \quad a.e.$

\end{ex}
\section{Construction of continuous frames}
For any separable Hilbert space there exists a frame  and more
generally any separable Banach space can be equipped with a Banach
frame with respect to an appropriately chosen sequence space
\cite{casazza}. Concerning the existence of continuous frames, it is
natural to ask: dose there exist continuous frames for any Hilbert
space and any measure space? The existence of continuous frame
depends on the dimension of space and the measure of $\Omega$ which
we  derive at the following propositions. For the answer we consider
four cases:
\begin{itemize}
\item $\mu(\Omega)=\infty$ and $dim\h=\infty$;
\item $\mu(\Omega)<\infty$ and $dim\h<\infty$;
\item $\mu(\Omega)=\infty$ and $dim\h<\infty$;
\item $\mu(\Omega)<\infty$ and $dim\h=\infty$.
\end{itemize}

\begin{prop}
Let $(\Omega,\mu)$ be a $\sigma$-finite measure space with infinite
measure and $\h$ an infinite dimensional separable Hilbert space.
Then  there exists a continuous  Parseval frame $F :\Omega\to\h$
with respect to $(\Omega,\mu)$.
\end{prop}
\begin{proof} Since $\Omega$ is $\sigma$-finite, it can be
written as a disjoint union $\Omega=\bigcup_{k\in \NN}\Omega_k$ of
countably many subsets $\Omega_k\subseteq \Omega$ such that
$\mu(\Omega_k) <\infty$ for all $k\in\NN$. Without loss of
generality, assume that $\mu(\Omega_k)> 0$ for all $k$. Let
$\{e_k\}_{k\in\NN}$ be the orthonormal base for $\h$.  Define the
function $F :\Omega\to\h$ by
$$\omega\mapsto F(\omega)=\frac{1}{\sqrt{\mu(\Omega_k)}}e_k ,
\quad (\omega\in\Omega_k).
$$Then, for all $f\in\h$,
$$\int_{\Omega}|\langle
f,F({\omega})\rangle|^{2}\,d\mu(\omega)=\sum_{k\in
\NN}\int_{\Omega_k}|\langle f,F({\omega})\rangle|^{2}\,d\mu(\omega)=
\sum_{k\in \NN}|\langle f,e_k\rangle|^{2}=A\|f\|^2.$$
 Thus, $F$ is a continuous  Parseval frame.
\end{proof}

It is possible to find a frame for any separable Hilbert space and
consequently a continues frame for any separable Hilbert space.
\begin{prop}\label{3.1}
Let $(\Omega,\mu)$ be a $\sigma$-finite measure space with infinite
measure and $\h$ an infinite dimensional separable Hilbert space.
Then  there exists a continuous  frame $F :\Omega\to\h$ with respect
to $(\Omega,\mu)$.
\end{prop}
\begin{proof} Since $\Omega$ is $\sigma$-finite, it can be written as a disjoint
union $\Omega=\bigcup\Omega_k$ of countably many subsets
$\Omega_k\subseteq \Omega$ such that $\mu(\Omega_k) <\infty$ for all
$k\in\NN$. Without loss of generality, assume that $\mu(\Omega_k)>
0$ for all $k$. Let $\{f_k\}_{k\in \NN}$ be a frame for $\h$ with
bounds $A$ and $B$.  Define the function $F :\Omega\to\h$ by
$$\omega\mapsto F(\omega)=\frac{1}{\sqrt{\mu(\Omega_k)}}f_k ,
\quad (\omega\in\Omega_k).
$$Then, for all $f\in\h$,
$$\int_{\Omega}|\langle
f,F({\omega})\rangle|^{2}\,d\mu(\omega)=\sum_{k\in \NN}|\langle
f,f_k\rangle|^{2}.$$ So $$A\|f\|^2\leq\int_{\Omega}|\langle
f,F({\omega})\rangle|^{2}\,d\mu(\omega)\leq B\|f\|^2 \quad
(f\in\h).$$
 Thus, $F$ is a continuous  frame with
frame bounds $A$ and $B$.
\end{proof}

\begin{prop}
Let $(\Omega,\mu)$ be a  measure space with finite measure and $\h$
a finite dimensional  Hilbert space. Then there exists a continuous
frame ( Parseval frame) $F :\Omega\to\h$ with respect to
$(\Omega,\mu)$.
\end{prop}

\begin{proof} Let $dim\h=n$ and $\{f_k\}_{k=1}^N$ be a frame for
$\h$ and $\Omega=\bigcup_{k=1}^{N}\Omega_k$ where $\Omega_k\subseteq
\Omega$, $1\leq k\leq N$, $0<\mu(\Omega_k)<\infty$ and $\Omega_k$'s
mutually disjoint . Then $
F(\omega)=\frac{1}{\sqrt{\mu(\Omega_k)}}f_k , \omega\in\Omega_k$ is
a continuous  frame for $\h$. If we chose the orthonormal base
instance frame for $\h$, then $F$ is continuous Parseval frame.
\end{proof}

\begin{prop}
Let $(\Omega,\mu)$ be a $\sigma$-finite measure space with infinite
measure and $\h$ a finite dimensional  Hilbert space. Then there
exists a continuous frame $F :\Omega\to\h$ with respect to
$(\Omega,\mu)$.
\end{prop}
\begin{proof}
 Since $\Omega$ is $\sigma$-finite, it can be written
as a disjoint union $\Omega=\bigcup\Omega_k$ of countably many
subsets $\Omega_k\subseteq \Omega$ such that $\mu(\Omega_k) <\infty$
for all $k$. Without loss of generality, assume that $\mu(\Omega_k)>
0$ for all $k\in\NN$. Let $\{f_k\}_{k=1}^{\infty}$ ( it is possible
to find a frame with infinitely member for finite dimensional
Hilbert space) be a frame for $\h$ with bounds $A$ and $B$. Define
the function $F :\Omega\to\h$ by
$$\omega\mapsto F(\omega)=\frac{1}{\sqrt{\mu(\Omega_k)}}f_k ,
\quad(\omega\in\Omega_k).
$$Then, for all $f\in\h$,
$$\int_{\Omega}|\langle
f,F({\omega})\rangle|^{2}\,d\mu(\omega)=\sum_{k\in \NN}|\langle
f,f_k\rangle|^{2}.$$ So $$A\|f\|^2\leq\int_{\Omega}|\langle
f,F({\omega})\rangle|^{2}\,d\mu(\omega)\leq B\|f\|^2\quad(f\in\h).$$
 Thus, $F$ is a continuous  frame with
frame bounds $A$ and $B$.
\end{proof}

In the case $\mu(\Omega)<\infty$ and $dim\h=\infty$, we have only
Bessel mapping.
\begin{prop}
Let $(\Omega,\mu)$ be a $\sigma$-finite measure space with finite
measure and $\h$ an infinite dimensional separable Hilbert space.
Then there exists a continuous Bessel mapping $F :\Omega\to\h$ with
respect to $(\Omega,\mu)$.
\end{prop}
\begin{proof} Let $\{f_k\}_{k=1}^\infty\infty$ be a Bessel sequence for $\h$
with bounds $A$, $B$ and $\Omega=\bigcup_{k=1}^{N}\Omega_k$ where
$\Omega_k\subseteq \Omega$, $1\leq k\leq N$,
$0<\mu(\Omega_k)<\infty$. Let $
F(\omega)=\frac{1}{\sqrt{\mu(\Omega_k)}}f_k , \omega\in\Omega_k$, then for all $f\in\h$,
$$\int_{\Omega}|\langle
f,F({\omega})\rangle|^{2}\,d\mu(\omega)=\sum_{k=1}^N|\langle
f,f_k\rangle|^{2}\leq B\|f\|^2$$
 Thus, $F$ is a continuous  Bessel mapping with
bound $B$.

\end{proof}
\section{Gabor and wavelet systems are continuous frames}\label{gaborandwavelet}

Well known examples for frames are wavelet and Gabor systems. The
corresponding continuous wavelet and STFT transforms give rise to
continuous frames. We make use of the following unitary operators on
$L^2(\RR)$:
\begin{itemize}
\item Translation: $T_x f(t) := f(t - x)$, for $f \in L^2(\RR)$ and $x \in
\RR$;
\item Modulation: $M_y f(t) := e^{2\pi i y\cdot t} f(t)$, for $f \in L^2(\RR)$ and $y \in
\RR$;
\item Dilation: $D_z f(t) := \frac{1}{|z|^{\frac{1}{2}}} f(\frac{t}{z})$, for $f \in L^2(\RR)$ and $z\neq
0$.
\end{itemize}
\begin{defn}
Let $ \psi \in L^{2}(\RR)$ be admissible, i.e.,
$$C_{\psi}:=\int_{-\infty}^{+\infty}\frac{|\hat{\psi}(\gamma)|^{2}}
{|\gamma|}\,d\gamma< +\infty.$$ For $a,b\in \RR$ that $a\neq0$, let
$$\psi^{a,b}(x):= (T_{b}D_{a}\psi)(x)= \frac{1}{|
a|^{\frac{1}{2}}}\psi(\frac{x-b}{a}), \quad ( x\in \RR).$$ Then the
continuous wavelet transform $W_{\psi} $ is
 defined by
$$W_{\psi}(f)(a,b):=\langle
f,\psi^{a,b}\rangle=\int_{-\infty}^{+\infty}f(x)\frac{1}{|
a|^{\frac{1}{2}}}\overline{\psi(\frac{x-b}{a})}\,dx, \quad (f \in
L^2(\RR)). $$
\end{defn}

For an admissible function $\psi$ in $L^2$, the system
$\{\psi^{a,b}\}_{a\neq0, b\in \RR}$ is a continuous tight frame for
$ L^{2}(R )$ with respect to $ \Omega = \RR\setminus\{0\}\times \RR
$ equipped with the measure $\frac{dadb}{a^{2}}$ and for all $f\in
L^{2}(\RR)$
$$f=\frac{1}{C_{\psi}}\int_{-\infty}^{+\infty}\int_{-\infty}^{+\infty}W_{\psi}(f)(a,b)
\psi^{a,b}\,\frac{dadb}{a^{2}},$$ where the integral is understood
in weak sense. This system constitutes a continuous tight frame with
frame bound $\frac{1}{C_{\psi}}$. If $\psi$ is suitably normed so
that $C_{\psi} = 1$, then the frame bound is $1$, i.e. we have a
continuous Parseval frame. For details, see the Proposition 11.1.1
and Corollary 11.1.2 of \cite{C4}.

\begin{defn}\label{D:Def_STFT}
Fix a function $ g\in L^{2}(\RR)\setminus\{0\}$. The
\textit{short-time Fourier transform} (STFT) of a function $f\in
L^{2}(\RR)$ with respect to the window function $g $ is given by
$$
\Psi_{g}(f)(y,\gamma)=\int_{-\infty}^{+\infty}f(x)\overline{g(x-y)}e^{-2\pi
i x\gamma}dx, \quad\quad( y, \gamma \in\RR).$$
\end{defn}
Note that in terms of modulation operators and translation
operators, $\Psi_{g}(f)(y,\gamma)=\langle f,
M_{\gamma}T_{y}g\rangle$.

Let $ g\in L^{2}(\RR)\setminus\{0\}$. Then $\{M_{b}T_{a}g\}_{a,b\in
\RR}$is a continuous frame for $L^{2}(\RR)$ with respect to $ \Omega
=\RR^{2}$ equipped with the Lebesgue measure. Let $f_{1}, f_{2},
g_{1}, g_{2}\in L^{2}(\RR)$. Then
$$\int_{-\infty}^{+\infty}\int_{-\infty}^{+\infty}\Psi_{g_{1}}(f_{1})(a,b)\overline{\Psi_{g_{2}}(f_{2})(a,b)}dbda
=\langle f_{1}, f_{2}\rangle\langle g_{2},g_{1}\rangle.$$ So this
system represent a continuous tight frame with bound $\|g\|^2$. For
details see the proposition 8.1.2 of \cite{C4}.
\par
Another example of continuous frames, called wave packets, can be constructed
by combinations of modulations, translations and dilations to
interpolate the time-frequency properties of analysis of Gabor and
wavelet frames. The interested reader can refer to \cite{ChRa}.
\\
\textbf{Acknowledgments:} The authors would like to thank referee(s) for valuable comments and suggestions.

\end{document}